\documentclass[11pt]{article}
\usepackage{geometry}
\usepackage{amsmath,amssymb,amsthm,xcolor}
\usepackage{tikz}
\usepackage{comment}
\usepackage{authblk}

\newcommand{\bN}{\mathbb{N}}

\newcommand{\bR}{\mathbb{R}}
\newcommand{\trap}{\mathcal{T}}
\newcommand{\Riem}{\operatorname{Riem}}
\newcommand{\Ric}{\operatorname{Ric}}

\newtheorem{thm}{Theorem}[section]
\newtheorem{prop}[thm]{Proposition}
\newtheorem{lem}[thm]{Lemma}
\newtheorem{conj}[thm]{Conjecture}

\newtheorem*{exintro}{Example {\ref{ex:AdS}}}
\theoremstyle{definition}
\newtheorem{defn}[thm]{Definition}
\newtheorem{exam}[thm]{Example}
\newtheorem{rem}[thm]{Remark}

\title{On the affine parametrization of null geodesics in low regularity}
\author[1]{Sa\'ul Burgos}
\author[2,3]{Leonardo Garc\'ia-Heveling}
\author[3]{Melanie Graf}
\affil[1]{Departamento de Geometr\'ia y Topolog\'ia, Facultad de Ciencias and IMAG, Universidad de Granada, Spain}
\affil[2]{Faculty of Mathematics, University of Vienna, Austria}
\affil[3]{Department of Mathematics, University of Hamburg, Germany}

\date{\today}

\begin{document}

\maketitle

\begin{abstract}
    In low-regularity spacetimes and Lorentzian length spaces, achronal causal curves play the role of null (pre-)geodesics. Because of the lack of a geodesic equation, they do not come with a canonical parametrization. In this context, we discuss a notion of affine parametrization via limits of affinely parametrized timelike geodesics. However, we point out a major drawback: an example where this approximation procedures gives a non-unique result, in a way that even completeness or incompleteness of the limit null geodesic is not well-defined. This example involves discontinuous gluing of two Lorentzian metrics across a null hypersurface to give a well-behaved Lorentzian length space and as such is, just like the parametrization problem itself, manifestly Lorentzian. In view of this, we explore possibilities for a Penrose-type singularity theorem using timelike geodesics.
\end{abstract}

\section{Introduction}

Null geometry plays a very important role in general relativity, as exemplified by two ce\-lebrated results: Penrose's singularity theorem infers the formation of singularities from gravitational collapse by establishing the existence of an incomplete null geodesic \cite{PenSing}. Hawking's area theorem encodes the second law of black hole thermodynamics as a statement about the area of cross sections of the event horizon, which is a null hypersurface \cite{HawArea}. These complement other important results, which are centered around timelike geodesics and spacelike hypersurfaces, such as Hawking's singularity theorem \cite{HawSing}.

A current trend in Lorentzian geometry is to consider results like the ones above under lower differentiability of the metric tensor, or even in synthetic frameworks that do away with the manifold structure altogether \cite{CMMsynth,CaMo,MelanieC1,Kett,KuSa,McCnull,SaSt}. In this paper, we adopt the definition of Lorentzian length spaces from Kunzinger and S\"amann \cite{KuSa} (in essence, metric spaces equipped with chronological and causal relations and a time separation function), but the issues that we raise are relevant to other approaches as well.

  In general, results involving timelike reasoning are more amenable to the non-smooth setting than those involving null geometry as they are more similar to their Riemannian counterparts. In both the timelike and the null case, the concept of a (pre-)geodesic can be defined as a causal (local) length maximizer, which in the null case is a (locally) achronal causal curve. An affine parametrization of a timelike geodesic can then be defined as a parametrization proportional to Lorentzian arclength, requiring no additional regularity of the space. The affine parametrization of a null geodesic, on the other hand, does not admit a straightforward characterization without invoking the geodesic equation. The latter is well-posed for spacetimes $(M,g)$ with $g \in C^{1,1}$. One can go down to $g \in C^1$ if willing to accept non-uniqueness of solutions \cite{MelanieC1}, and even $g \in C^{0,1}$ if considering solutions in the sense of Filippov \cite{SaSt}. In any of these regularities one can determine a parametrization in which local maximizers will become solutions of the geodesic equation (cf.\ \cite[Thm.\ 1.1]{LangeLytchakSaemann}), but for metrics below $C^{1,1}$ there may arise further solutions to the geodesic equation which are not locally length maximizing. 
  For Lorentzian length spaces, however, there is no analogue of the geodesic equation at all.

In synthetic settings, Hawking's singularity theorem for measured Lorentzian length spaces was already proven in the work of Cavalletti and Mondino, where these spaces were first introduced \cite{CaMo}. Their result is based on an optimal transport formulation of the strong energy condition,
\begin{equation*}
    \Ric(X,X) \geq 0 \quad \text{for all timelike } X \in TM.
\end{equation*}
The weaker null energy condition appearing in Penrose's singularity theorem and Hawking's area theorem,
\begin{equation*}
    \Ric(X,X) \geq 0 \quad \text{for all null } X \in TM,
\end{equation*}
has been given three different optimal transport characterizations by McCann \cite{McCnull}, Ketterer \cite{Kett}, and Cavalletti, Manini, Mondino \cite{CMMsynth,CMMsmooth}. Of these, only the last one has so far provided a version of Penrose's theorem for Lorentzian length spaces. None of the three, however, solves the issue of defining the affine parametrization of a null geodesic (in the sense of achronal causal curve) in a (measured) Lorentzian length space. In particular, this issue is bypassed in \cite{CMMsynth} by encoding the parametrization in a ``gauge function" that has to be specified as part of the data of a synthetic null hypersurface. The question remains if there is a canonical way to construct such a gauge function using only the structure of the ambient Lorentzian length space.

Our paper deals precisely with this question. On smooth spacetimes, one can view (affinely parametrized) null geodesics as the limits of timelike ones. This seems a promising way to proceed in the synthetic setting, and is adopted for geodesics with endpoints in \cite{McCnull}. In the context of Penrose's singularity theorem, however, it is essential to deal also with null geodesics without an endpoint. Their behavior is much more subtle, as shown by the following example.

\begin{exintro}
    A globally hyperbolic, regularly localizable Lorentzian length space $X$ with 
    timelike curvature bounded above by $-1$ (in the sense of  of \cite[Def 3.2]{CurvBoundsLLS}), containing an achronal causal curve $\gamma \colon [0, \infty) \to X$ such that:
    \begin{enumerate}
        \item In this parametrization $\gamma$ can be approximated uniformly on compact sets by length-maximizing timelike geodesics.
        \item $\gamma$ admits a reparametrization $\sigma \colon [0,a) \to X$ with $a < \infty$, such that $\sigma$ too can be approximated uniformly on compact sets by length-maximizing timelike geodesics.
    \end{enumerate}
\end{exintro}

In this example, the approximation procedure does not yield a unique parametrization (more precisely, not unique up to affine reparametrization, which is the most we could expect), but even worse: the completeness or incompleteness of $\gamma$ is not well-defined. The example comes in the form of a spacetime where the metric has been multiplied by a discontinuous conformal factor. As already stated above, the space $X$ enjoys many good properties, so in the world of Lorentzian length spaces it is not apparently pathological. The only perhaps desirable property that we are not able to establish is a lower curvature bound: To the contrary we show that it does not have curvature bounded below by 0, which due to the curvature bounds of the two glued spacetimes would be the natural candidate, leading us to suspect that it not have any lower bound on the curvature at all. 


In view of our example, one wonders if it is possible to avoid the issue of affine parametrization for null geodesics altogether and instead prove a version of Penrose's theorem using timelike geodesics instead of null ones. Indeed, what sets the theorem apart is the assumption of a trapped surface, interpreted as gravitational collapse inside a black hole (as opposed to an expanding Cauchy surface in Hawking's theorem, interpreted as a cosmological condition). The conclusion of Penrose's theorem being null (rather than timelike) geodesic incompleteness is a consequence of the techniques employed in the proof, but in fact less desirable, given that the affine length of timelike geodesics has a direct physical interpretation as proper time.

Following these ideas, in Section~\ref{sec:timelikepenrose} we obtain a proof of the classical Penrose theorem on smooth spacetimes, using only approximation by timelike geodesics and the reformulation of the null energy condition as variable lower timelike Ricci curvature bounds of McCann in 
\cite{McCnull}, and thus avoiding direct reasoning with null geodesics. To our knowledge, this is the first application of this reformulation of the null energy condition. 
We also obtain a purely timelike version of Penrose's theorem, predicting timelike incompleteness from a trapped surface and a condition on timelike Ricci and sectional curvatures, improving upon related results obtained in \cite{LeoPRD}.

\section{Null geodesics as limits of timelike ones}

In this section, we explore how one can define the parametrization of null geodesics as limits of timelike ones, especially with regards to obtaining a notion of completeness or incompleteness.
Since the theory of Lorentzian pre-length spaces is by now well established, we refer to \cite{KuSa} for the fundamentals. 

Throughout this section, $(X, d, \ll, \leq, \tau)$ is a globally hyperbolic regular Lorentzian length space \cite[Def.~3.22]{KuSa}. In this context, we call a causal curve $\gamma \colon I \to X$ a \emph{timelike geodesic} if it is timelike and $\tau(\gamma(s),\gamma(t)) = L_\tau(\gamma \vert_{[s,t]})$ for all $s < t \in I$ (if $I$ is a compact interval $[a,b]$, it suffices to consider $s=a$, $t=b$). Similarly, we will call $\gamma$ a \emph{null geodesic} if it is achronal (which automatically implies that it is maximizing with vanishing length). Regularity of $X$ means that maximizers always have a well-defined causal character, so every maximizing causal curve will fit into one of the two types. Note that our definition of geodesic is parametrization invariant. We will use the term \emph{affinely parametrized timelike geodesic} for timelike geodesics that are parametrized proportional to Lorentzian arclength, and \emph{affinely parametrized null geodesic} for null geodesics satisfying some tentative definition of affine parametrization. 

\subsection{Approximation results}

We start by showing that a null geodesic $\gamma$ can be approximated by timelike geodesics $\gamma_n$. This is, of course, a necessary condition in order to define an affine parametrization of $\gamma$ via a limiting procedure. At this stage, however, we are not claiming that the $\gamma_n$ can be chosen affinely parametrized, since our proof requires $(\gamma_n)_n$ to be an equi-Lipschitz sequence with respect to the background distance $d$, and these two conditions are incompatible in general (more on that below). First we deal with null geodesics having two endpoints.

\begin{prop} \label{prop:approx}
     Let $p,q \in X$ and suppose that there is a unique (up to parametrization) null geodesic $\gamma \colon [0,1] \to X$ from $p$ to $q$. Then, a suitable reparametrization of $\gamma$ is the uniform limit of a sequence of timelike geodesics $\gamma_n \colon [0,1] \to X$. 
\end{prop}

\begin{proof}
    By the localizability property of Lorentzian length spaces, there is a sequence $q_n \to q$ such that $q \ll ... \ll q_n \ll q_{n-1} \ll ... \ll q_0$. By the push-up property, $q_n \in I^+(p)$, and by global hyperbolicity and regularity, there is a maximizing timelike geodesic $\gamma_n \colon [0,1] \to X$ from $p$ to $q_n$ (parametrized proportional to $d$-arclength). By global hyperbolicity, and since the $\gamma_n$ are contained in the compact causal diamond $J(p,q_0)$, the lengths of the $\gamma_n$ are uniformly bounded above by a constant $L>0$. Then the $\gamma_n$ are $L$-Lipschitz, and the limit curve theorem \cite[Thm.~3.7]{KuSa} tells us that a subsequence of $(\gamma_n)_n$ converges to a causal curve with endpoints $p$ and $q$, which by assumption (when $\gamma$ is unique) can only be $\gamma$, up to parametrization.

\end{proof}

Next, we deal with curves that have only one endpoint, which is necessary in order to talk about geodesic (in)completeness.

\begin{prop} \label{prop:approxx}
    Suppose that $X$ is past null non-branching and let $\gamma \colon [0,\infty) \to X$ be a null geodesic. Then there exists a sequence of timelike geodesics $\gamma_n \colon [0,n] \to X$ such that $(\gamma_n)_n$ converges uniformly on compact subsets to a suitable reparametrization of $\gamma$.
\end{prop}

Here by past null non-branching we mean that for any two null geodesics $\gamma,\sigma \colon [0,1] \to X$ it holds:
\begin{equation*}
    \gamma(s) = \sigma(s) \quad \forall s \in \big[\tfrac{1}{2},1\big] \implies \gamma([0,1]) \subseteq \sigma([0,1]) \ \text{ or } \ \sigma([0,1]) \subseteq \gamma([0,1]).
\end{equation*}
The choice of the intervals is arbitrary, and since we did not assume a specific parametrization, all that this really means is that if the images of $\gamma$ and $\sigma$ overlap on a segment, then they coincide (up to one of them being possibly longer). Note that \cite[Def.~5.6]{CMMsynth} provides a related definition of null non-branching tailored to null geodesics contained in a synthetic null hypersurface.

\begin{proof}
    First, we prove that for all $n \in \bN$, $\gamma \vert_{[0,n]}$ is the unique causal curve from $\gamma(0)$ to $\gamma(n)$. Suppose that there was another causal curve $\sigma$ with the same endpoints. Then we can concatenate $\sigma$ with $\gamma \vert_{[n,n+1]}$ to obtain a causal curve $\sigma * \gamma \vert_{[n,n+1]}$ from $\gamma(0)$ to $\gamma(n+1)$, which is necessarily a null geodesic, given that $\gamma(0)$ and $\gamma(n+1)$ are not timelike related. But then $\gamma\vert_{[0,n+1]}$ and $\sigma * \gamma \vert_{[n,n+1]}$ contradict the past null non-branching assumption.
    
    We can now apply Proposition~\ref{prop:approx} to $\gamma \vert_{[0,n]}$ (affinely reparametrizing between $[0,1]$ and $[0,n]$ when needed), to conclude that for every $n \in \bN$, there is a timelike geodesic $\gamma_n \colon [0,n] \to X$ with $\gamma_n(0) = \gamma(0)$ and such that $d(\gamma(s),\gamma_n(s)) \leq \frac{1}{n}$ for all $s \in [0,n]$. We have thus constructed the desired sequence $(\gamma_n)_n$.
\end{proof}

Finally, one can also consider 
geodesics with no endpoints, but we shall omit this for brevity.

\begin{rem}[Smooth case]
 First of all, note that in the case of a smooth spacetime, a solution of the geodesic equation is a geodesic in the sense defined above only if it is also maximizing. Nonetheless, one could prove Propositions~\ref{prop:approx} and \ref{prop:approxx} using well-posedness of the initial value problem for the geodesic equation, thus obtaining even a $C^1$ approximation (as is done below to prove the similar Proposition~\ref{prop:approxtrap}). Global hyperbolicity ensures the relevant maximization properties of geodesics via the tools found in \cite[Chap.~9]{BEE} (in particular, Prop.~9.33 therein). On the other hand, if one does not insist on the geodesics being maximizing, then such a proof is possible without assuming global hyperbolicity.
\end{rem}

Let us now turn to discussing the issue of affine parametrizations. Ideally, any generalization of affine parametrizations to geodesics in Lorentzian length spaces satisfies two core principles (in addition to coinciding with the usual notion in the smooth setting):
\begin{enumerate}
    \item Existence: Any causal geodesic admits an affine parametrization.
    \item Uniqueness: Affine parametrizations of a given geodesic are unique up to affine reparametrizations.
\end{enumerate}
For timelike geodesics on Lorentzian length spaces, one defines that an affine parametrization is one that is proportional to $\tau$-arclength, in analogy to the smooth case. It is well-known that this satisfies the desired existence and uniqueness properties (by Proposition~3.34 and Lemma~2.28 in \cite{KuSa}, respectively).

The above propositions suggest that we could define the affine parametrization of a null geodesic as follows: For a null geodesic $\gamma$, an affine parametrization is a parametrization obtained as a limit of any sequence of affinely parametrized timelike geodesics $(\gamma_n)_n$ converging to $\gamma$. We know that given an initial parametrization of $\gamma$, we can find an approximation of $\gamma$ by timelike geodesics $\gamma_n$ (not necessarily affinely parametrized). We can then reparametrize the $\gamma_n$ proportional to $\tau$-arclength to obtain a new sequence $\tilde\gamma_n$, and hope that the $\tilde\gamma_n$ converge to a reparametrization of $\gamma$ (which we would then call an affinely parametrized null geodesic). The $\tilde \gamma_n$, however, are not guaranteed to be locally Lipschitz (see \cite[Sec.~3.7]{KuSa} and \cite[Cor.~6]{McCnull}). Hence we cannot apply the usual limit curve theorem to the $\tilde\gamma_n$. So even existence might be non-trivial. We refrain from trying to solve this technical issue, since in the next section we will see that even when this problem does not appear, there are issues with {\em uniqueness} causing a more conceptual problem relating to (in)completeness. Note that, in any case, on smooth spacetimes the above procedure does recover the affine parametrization of the null geodesic (one can deduce this from well-posedness of the initial value problem for the geodesic equation).

Given our example it becomes reasonable to ask if one will have to give up on either existence or uniqueness of affine parametrizations for null geodesics (defined as locally achronal curves) in Lorentzian length spaces. Another possibility would be to restrict oneself to only consider those locally achronal curves which admit a suitable parametrization as ``null geodesics'' in the first place (cf.\ Remark \ref{rem:McCannnullgeods} below). Then this could either be further restricted to only allow curves for which this parametrization is suitably unique (which does not seem desirable, as it excludes the curve from our example, which we believe should be considered a null geodesic) or to treat parametrizations which differ by more than an affine change of variables as two {\em distinct} parametrized geodesics (i.e., the curve from our example would give rise to two parametrized geodesics with different completeness properties). It is at this point unclear what influence either of these approaches would have on a synthetic Penrose theorem.


\begin{rem}[Connection to \cite{McCnull}]\label{rem:McCannnullgeods}
 In McCann's terminology, the space $\operatorname{CGeo}^\ell(M)$ of \emph{causal $\ell$-geodesics} is defined essentially as described above. More precisely, $\operatorname{CGeo}^\ell(M)$ is the closure of the space $\operatorname{TGeo}^\ell(M)$ of \emph{timelike $\ell$-geodesics},\footnote{Note that being a timelike $\ell$-geodesic includes being parametrized by Lorentzian arc length.} and an element of $\operatorname{CGeo}^\ell(M)$ that is not in $\operatorname{TGeo}^\ell(M)$ is called a \emph{lightlike $\ell$-geodesic} (see also \cite[Rem.~9]{McCnull}). To our knowledge, however, it is not known if every achronal causal curve admits a reparametrization that is contained in $\operatorname{CGeo}^\ell(M)$, or if any pair of causally related points can be joined by a curve in $\operatorname{CGeo}^\ell(M)$. We expect that the same difficulty related to (non-)Lipschitzianity appears.
\end{rem}

\subsection{Null geodesic (in)completeness}

As discussed above, it is tempting to try to define the affine parametrization of a null geodesic via a limiting procedure. While this coincides with the usual affine parametrization in the smooth case, which is unique up to translations and linear scalings, this uniqueness is not at all clear in the non-smooth setting as different approximating sequences may lead to fundamentally different parametrizations. Even worse, it is not even clear if completeness or incompleteness (i.e.\ whether the parameter range is bounded or unbounded) depends on the approximating sequence. In fact our main example, Example~\ref{ex:AdS} below, shows that it does. 

Before we get to that example, let us first remark with a simple example that 
even in the smooth case, there is no straightforward relationship between the limit of the lengths of the approximating timelike geodesics provided by Proposition~\ref{prop:approxx} and the (in)completeness of the null geodesic.

\begin{exam}
    Consider in $2$-dimensional Minkowski spacetime $(M,g)$ with coordinates $(t,x)$ the null geodesic segment $\gamma \colon [0, \infty) \to M,\ s \mapsto (s,s)$. Consider the sequences of timelike geodesics
    \begin{equation*}
        \gamma_n \colon [0,b_n) \longrightarrow M,\quad s \longmapsto \left(s,s \sqrt{1-\frac{1}{n^2}}\right),
    \end{equation*}
    where we consider different options for $b_n$. Note that the segments $\gamma_n$ are maximizing, so their length is given by
    \begin{equation*}
        L_n = b_n^2 - \left(1 - \frac{1}{n^2}\right) b_n^2 = \frac{b_n^2}{n^2}.
    \end{equation*}
    Hence if we choose $b_n = n^a$, we have $L_n \to 0$ if $0 < a < 1$, $L_n = 1$ if $a = 1$, and $L_n \to \infty$ if $a>1$. Nonetheless, in all these cases $b_n \to \infty$, and the sequence $\gamma_n$ approximates the entire curve $\gamma$.
\end{exam}

This shows that trying to shortcut the discussion of (in-)completeness of null geodesics by simply considering lengths of approximating timelike geodesics will not work and one really has to solve the parametrization problem. This is where our next examples comes in.

\begin{figure}
    \centering
    \includegraphics[scale=5]{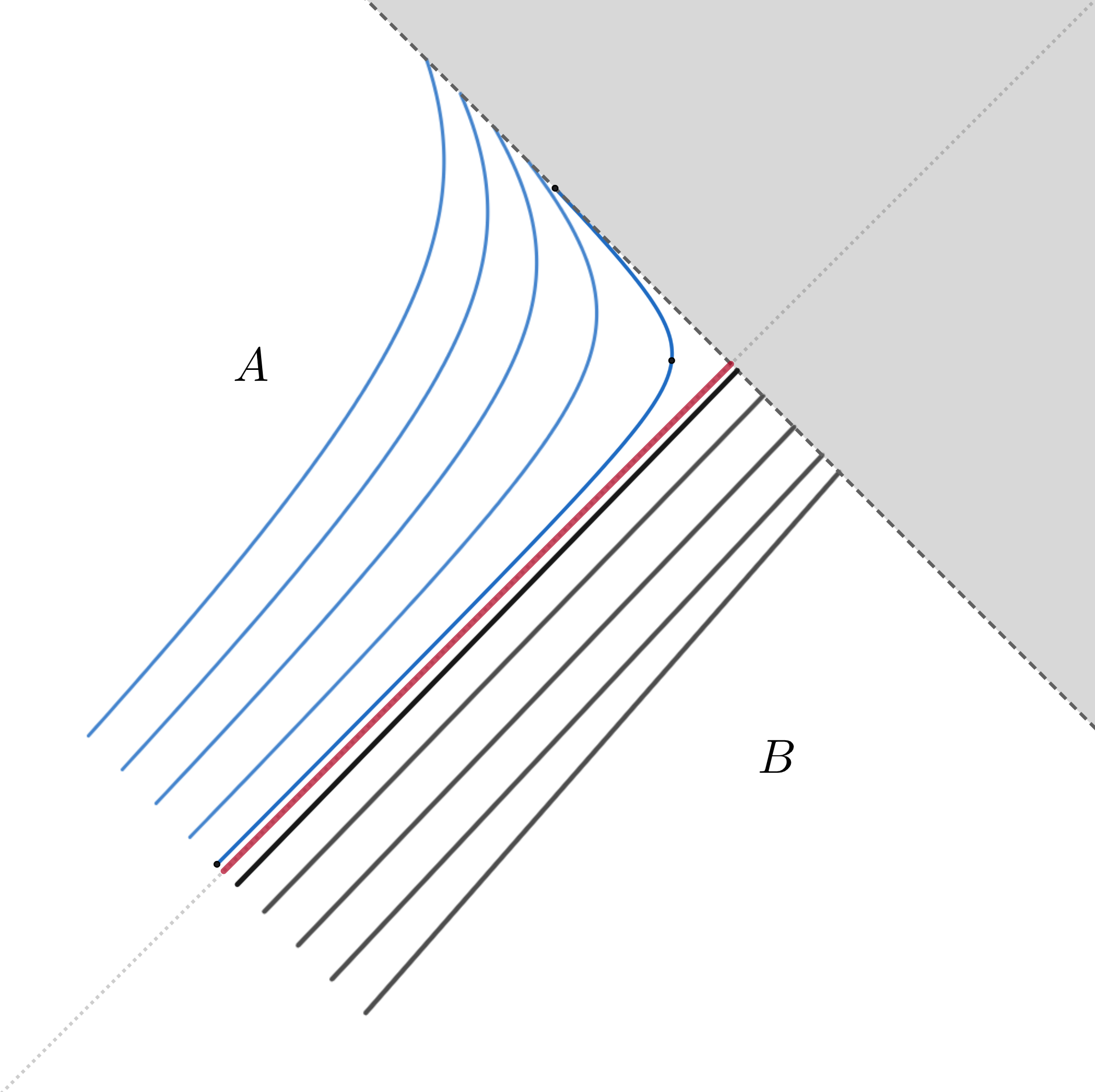}
    \caption{The space $X$ of Example~\ref{ex:AdS}, with the null geodesic $\sigma$ in red, and the approximating sequences from regions $A$ and $B$.}
    \label{fig:AdS}
\end{figure}

\begin{exam} \label{ex:AdS}
    Let $M := \{ (t,x) \mid t<-x \}$ and $g:=\Omega(-dt^2+dx^2)$, with discontinuous conformal factor
    \begin{equation*}
        \Omega := \begin{cases}
            1 &\text{if } \vert t \vert \geq \vert x \vert, \\
            \frac{1}{x^2} &\text{if } \vert t \vert < \vert x \vert.
        \end{cases}
    \end{equation*}
    Despite $\Omega$ being discontinuous, $g$ still induces the structure of a Lorentzian length space (see Lemma~\ref{lem:XisLLS} below). 
    
    We prove that the null geodesic $\sigma$ with image $\{ t=x \}$ can be given both a complete and an incomplete affine parametrization, depending on whether one chooses approximating affinely parametrized timelike geodesics lying above or below it.
    
    Notice that the region $A := \{ \vert t \vert < -x \}$ (where $\Omega \not\equiv 1$) is isometric to a patch of anti-de Sitter space \cite[Eqn.~(12)]{AdS}, and can be isometrically embedded in the spacetime $N = \bR^2$ with metric
    \begin{equation*}
        h := - e^{2y} dt^2 + dy^2.
    \end{equation*}
    The isometric embedding is given by $\phi : (t,x) \mapsto (t, -\ln(-x))$. To find the null geodesics $\gamma(s) = (t(s),y(s))$, we exploit the fact that $\partial_t$ is a Killing vector, so that it suffices to solve the following system
    \begin{equation*}
        \begin{cases}
            g(\dot\gamma,\partial_t) = - e^{2y} \dot{t} = -E, & \\
            g(\dot\gamma,\dot\gamma) = - e^{2y} \dot{t}^2 + \dot{y}^2 = 0, &
        \end{cases}
    \end{equation*}
    where $E>0$ for future-directed geodesics. The solution is given by
    \begin{equation*}
            t(s) = \frac{-1}{E(s-s_0)} + t_0, \qquad
            y(s) = \ln(\pm E(s-s_0)).
    \end{equation*}
    for constants $s_0,t_0$. Mapping back with $\phi^{-1}$, we obtain
    \begin{equation*}
            t(s) = \frac{-1}{E(s-s_0)} + t_0, \qquad
            x(s) = \frac{\pm 1}{E(s-s_0)}.
    \end{equation*}
    The solution thus exists for all $s>s_0$, and hence future directed null geodesics with $\dot{x}>0$ are future null geodesically complete in $(N,h)$.\footnote{We note that $(N,h)$ is not null geodesically complete as the half space coordinates $\left(\bR \times (-\infty,0),\frac{1}{x^2}(-dt^2+dx^2)\right)$ do not cover all of $AdS_2$ and in fact null geodesics will be incomplete as $x\to -\infty$.} It follows that we can parametrize $\sigma$ in $(M,g)$ as a future complete null geodesic, by approximating it with timelike geodesics in $A$ (note that $A$ is causally convex, so length maximizers in $A$ are still $\tau$-length maximizers in $M$). On the other hand, we can also approximate $\sigma$ with timelike geodesics in $B := \{ t < x \}$ (which is also causally convex), and then we obtain the future incomplete parametrization that $\sigma$ would have in Minkowski spacetime.
\end{exam}

Having found an example with an offending geodesic one might wonder if there is any a priori reason why such an example might be excluded for being ''pathological'' in some other way. However, as we will now see, despite the discontinuity in the metric this example is a very well behaved for a Lorentzian length space.

\begin{lem} \label{lem:XisLLS}
 The space $X$ in Example~\ref{ex:AdS} is a strongly regularly localizable Lorentzian length space. Moreover, it is globally hyperbolic and has timelike curvature bounded above by $-1$ in the sense of triangle comparison.
\end{lem}

\begin{proof}
First, we need to show that the time-separation function is lower semicontinuous, which can be done as in the smooth case \cite[Lem.~4.4]{BEE}. Indeed, let $p,q \in X$. If $\tau(p,q) = 0$, there is nothing to prove, so suppose $p \ll q$.
Let $\gamma \colon [0,1] \to X$ be a causal curve from $p$ to $q$ of length $L_g(\gamma) > \tau(p,q) - \varepsilon$. Clearly we can choose $\gamma$ such that it intersects $\partial A$ in at most a single point, and by slightly shifting that point if needed, we may assume that $\gamma$ is timelike everywhere (note that despite the discontinuity of $g$, $L_g$ is of course continuous under this shift). Then there is a $\delta > 0$ such that $L_g(\gamma \vert_{[\delta,1-\delta]}) > \tau(p,q) - 2\varepsilon$. Since $p \ll \gamma(\delta)$ and $\gamma(1-\delta) \ll q$ and $\ll$ is open, we find a neighborhood $U$ of $(p,q)$ in $X \times X$ such that $\tau(\tilde p, \tilde q) \geq \tau(p,q) - 2\varepsilon$ for all $(\tilde p, \tilde q) \in U$. Thus, $X$ is a Lorentzian pre-length space.

 That $X$ is a strongly localizable Lorentzian length space then follows analogously to the proof that spacetimes with a continuous, strongly causal and causally plain metric are strongly localizable Lorentzian length spaces in \cite[Thm.\ 5.12]{KuSa}:  First note that both $g$-causal and $\leq$-causal curves always decompose into two segments, one contained in $B$ and the other in $\bar{A}$, where $g$ is smooth. Hence $g$-causal and $\leq$-causal curves still coincide. Together with additivity of both the $g$-length and the $\tau$-length this also immediately implies that \cite[Lem.\ 5.10]{KuSa} still holds and it only remains to establish strong localizability.  Strong localizability is a technical condition, but essentially means that every point has a basis of neighborhoods such that the $d$-length of causal curves in these neighborhoods is bounded and which, with the induced causal relation, have the structure of a Lorentzian length space with continuous time separation. For this we note that $(X,g)$ has the same causal structure as Minkowski, hence is globally hyperbolic. Running the usual limit curve arguments 
 or using the limit curve theorem for locally causally closed Lorentzian pre-length spaces, \cite[Thm.~3.7]{KuSa}, we see that maximizing causal curves always exist and so $\tau$ can be shown to be (globally) continuous as in \cite[Thm.\ 5.12]{KuSa}.  Again because the causal structure is that of Minkowski, every point has a basis of causally convex neighborhoods, which fulfill all the required properties (for the $d$-length bound we can again exploit that any causal curve splits into two segments, each of which is contained in a smooth spacetime).


 Regularity means that any maximizing curve between chronologically related points must be timelike (with no null segments). This is clear if both endpoints lie on the same side $A$ or $B$. Hence we deal with the case of a causal maximizer $\gamma$ with starting point $p\in B$
 and $q\in A$
 , where $p \ll q$. The curve $\gamma$ is then the concatenation of a curve $\gamma_B$ with image in $B$ and another curve $\gamma_A$ with imagine in $\overline{A}$. Both $\gamma_A$ and $\gamma_B$ need to be themselves maximizing, and therefore each of them is either everywhere timelike or everywhere null. There are thus three cases that we must rule out:
 \begin{itemize}
  \item Case 1: both $\gamma_A$ and $\gamma_B$ are null. Then $\gamma$ is null, but since $p \ll q$, it cannot be maximizing.
  \item Case 2: $\gamma_A$ 
   is null and $\gamma_B$ 
   is timelike. In this case, we compare with the LLS of Example~\ref{ex:Mink} for some $m\in \mathbb{R}$ satisfying
  \begin{equation*}
   \frac{1}{m^2} < \min\left\{ \frac{1}{x^2} \colon x \in \pi_x( I(p,q)) \right\},
  \end{equation*}
  where $\pi_x$ denotes the projection onto the $x$-coordinate. We denote the LLS of Example~\ref{ex:Mink} by $\tilde X$, the length in $X$ by $L$ and the length in $\tilde X$ by $\tilde L$. Now we conformally identify $X$ to $\tilde X$ by the identity map in the $(t,x)$-coordinates, and let $\sigma$ denote the $\tilde L$-maximizer between $p$ and $q$. We note that $\sigma $ is timelike, cf.\ Example~\ref{ex:Mink}. Then
  \begin{equation*}
   \tilde L (\sigma_A) < L (\sigma_A) \quad \text{and} \quad \tilde L (\sigma_B) = L (\sigma_B).
  \end{equation*}
  At the same time, $L(\gamma) = \tilde L(\gamma)$, given that $\gamma$ is null in $A$, and $X$ and $\tilde X$ are isometric on $B$. Then
  \begin{equation*}
   L(\gamma) = \tilde L(\gamma) \leq \tilde L (\sigma) = \tilde L(\sigma_A) + \tilde L(\sigma_B) < L(\sigma_A) + L(\sigma_B) = L(\sigma),
  \end{equation*}
  and $\gamma$ is not maximizing.
  \item Case 3: $\gamma_A$ is timelike and $\gamma_B$ is null. Let
  \begin{equation*}
   \ell(s) := \tau(p, r_s) + \tau(r_s, q),
  \end{equation*}
  where $r_s := (s,s)$. Thus the length of the maximizer $\gamma$ between $p$ and $q$ equals $\ell(s_0)$ for $s_0$ such that $r_{s_0}$ is precisely the intersection point of $\gamma$ with the separation between $A$ and $B$ (i.e.\ the null geodesic $s \mapsto (s,s)$). The case that $\gamma_B$ is null corresponds to $\tau(p,r_{s_0}) = 0$. We show that then, an infinitesimal increase in $s$ makes $\ell$ increase, contradicting the assumption that $\gamma$ maximizes. Since in the case at hand, $\tau(r_s,q) > 0$, and the time separation is smooth wherever positive (inside of the region $A$, where the metric is smooth), we conclude that $\frac{d}{ds} \vert_{s_0} \tau(r_s,q)$ is finite. On the other hand, $\frac{d}{ds} \vert_{s_0} \tau(p,r_s) = \infty$. To see this, let $p := (t,x) \in A$, so that
  \begin{equation*}
   \tau(p,r_s) = \sqrt{(t-s)^2 - (x-s)^2}.
  \end{equation*}
  But since we are assuming that $\tau(p,r_{s_0}) = 0$ (i.e. $s_0 = \frac{t+x}{2}$), it follows that the derivative explodes at $s_0$.
 \end{itemize}
So our space is regular and (strongly) localizable, which by \cite[Lemma 3.6]{CurvBoundsLLS} is equivalent to being regularly (strongly) localizable. 
 
Lastly, the fact that $X$ has curvature bounded above by $-1$ is a direct consequence of the gluing lemma for manifolds \cite[Lem.~4.3.3]{gluing}, since it is constructed by gluing two Lorentzian manifolds with curvature bounded above by $-1$ in the sense of Alexander and Bishop \cite{AlBi}.
\end{proof}

\begin{rem}
    Our example is two-dimensional, however one can always construct higher dimensional examples by taking a product with $\bR^n$. The only part of the above discussion that does not go
    through in that case is the argument showing that curvature is bounded above by $-1$. This is because curvature bounds in the sense of Alexander and Bishop  are not preserved when taking products. Specifically, when taking a product with $\bR^n$, there are spacelike planes with sectional curvature $R = 0$.\footnote{A Lorentzian manifold is said to have (sectional) curvature bounded above by $K$ in the sense of \cite{AlBi} if spacelike sectional curvatures are $\leq K$ and timelike sectional curvatures are $\geq K$.} 
\end{rem}

\begin{rem}
    In their recent work \cite{CMMsynth}, Cavalletti, Manini and Mondino define a synthetic null hypersurface to be a closed achronal subset equipped with a ``gauge" function $G$ and a measure $\mathfrak{m}$. For a smooth null hypersurface $S$, there is a canonical choice of these structures induced by the ambient spacetime (up to a natural equivalence relation of measures and gauges, and provided that $S$ admits a global section, see also \cite{CMMsmooth}). For a closed achronal subset in a general Lorentzian length space, however, it is not possible to directly adapt the same construction. In fact, in our Example~\ref{ex:AdS}, there seems to be no such canonical choice: One can construct two non-equivalent pairs of gauge and measure on $\{ t=x \}$ by viewing it as a smooth null hypersurface in Minkowski spacetime $(\bR^2,-dt^2+dx^2)$ or in anti de Sitter spacetime in half space coordinates $\left(\bR \times (-\infty,0),\frac{1}{x^2}(-dt^2+dx^2)\right)$. Notice that the null energy condition in the sense of \cite{CMMsynth} is satisfied for both choices, because it is satisfied (in the usual sense) in both ambient spacetimes.
\end{rem}

As mentioned, Lemma \ref{lem:XisLLS} shows that Example \ref{ex:AdS} is a reasonable Lorentzian length space. However, our space $X$ does not have a lower curvature bound by $0$ in the sense of one-sided timelike triangle comparison \cite[Def 3.2]{CurvBoundsLLS}. This is in spite of the fact that the regions $A$ and $B$ on their own do satisfy such a bound, since they are isometric to patches of AdS and Minkowski spacetime, respectively. This (together with performing analogous computations for other explicit example triangles and  potential weaker lower curvature bounds) leads us to strongly  suspect that this space will not have any lower curvature bound at all, though we do not have a proof at the moment.

\begin{rem}[No lower curvature bound by $0$]
To prove that our space $X$ does not have a lower curvature bound by $0$ in the sense of one-sided timelike triangle comparison, we construct a timelike triangle which fails to achieve the required bounds. Any such triangle cannot have all its vertices in either $\bar{B}$ or $\bar{A}$. The simplest choice is therefore to have one vertex in $B$, one on the boundary and one in $A$, as depicted in Figure \ref{fig:ads_mink}.

Let us turn to the details: In anti-de Sitter space, seen as the hyperboloid embedded in $\bR ^3_2$, using \cite[Eqn.~(2.7)]{CRHK17} together with \cite[Lem. 4.24]{Oneill} (as in the proof of \cite[Lem.~4.27]{CurvBoundsLLS}), we have that for causally related points $v,w$ the Lorentzian distance is $\bar\tau (v,w) = \arccos \left(- \langle v,w \rangle \right)$. For our model, we have the isometric embedding $\phi: \{|t| < |x| \} \to AdS_2 \subset \bR ^3_2$ given by
\begin{equation*}
    X_1 = \frac{1}{2x} (1 +x^2 - t^2), \quad X_2= \frac{t}{x}, \quad X_3 = \frac{1}{2x}(1-x^2 + t^2).
\end{equation*}
Thus for any $p = (t_p, x_p) \ll q = (t_q, x_q)$ in $\{ |t| \leq |x|\}$, the Lorentzian distance is
\begin{equation*}
    \tau (p,q) = \arccos\left( \frac{x_q^2 + x_p^2 - (t_q - t_p)^2}{2x_p x_q} \right).
\end{equation*}

Solving the geodesic equation in this region, assuming that timelike geodesics are parametrized by arclength, we get that in the region $A = \{|t| < |x|\}$, any timelike geodesic $\gamma (s) = \left( t(s) , x(s) \right)$ is given by
\begin{equation*}
    t(s) = \frac{1}{E} \tan (s - s_0)+t_0, \qquad x(s) = -\frac{1}{E} \sec(s - s_0),
\end{equation*}
where $- E = g (\partial_t, \dot \gamma)$ for future geodesics since $\partial_t$ is Killing, with $E, s_0 > 0$ and $t_0$ constants. 
So for any $q = (t_q , x_q)$, $r = (t_r, x_r) \in \{|t| \leq |x|\}$ we have that the timelike geodesic $\gamma_{qr}:[0,\tau (q,r)] \to M$ from $q$ to $r$ satisfies:
\begin{equation*}
    t_0 = \frac{(t_r^2 - t_q^2) - (x_r^2 - x_q^2 )}{2(t_r-t_q)}, \quad E = \frac{1}{\sqrt{x_q^2 - (t_q - t_0)^2}}, \quad \text{and} \quad s_0 = \arccos\left( -\frac{1}{Ex_q} \right).
\end{equation*}

In order to compute the Lorentzian distance from $p \in B = \{t < x \}$ to $r \in A = \{|t|<|x|\}$ we need to find a maximizing curve between the two points. This curve will cross the null pregeodesic $\{t = x\}$ at a point $(x,x)$. Thus the length of such curve can be computed by maximizing, using elementary calculus, the length function
\begin{equation*}
\ell_{pr} (x) = \sqrt{(x - t_p)^2 - (x - x_p)^2} + \arccos\left( \frac{x_r^2 + x^2 - (t_r - x)^2}{2x x_r} \right).
\end{equation*}
For the triangle $p = (-3/2,-1), q= (-3/4, - 3/4), r= (0,-1)$ we have that
\begin{equation*}
    \tau (p,q)  = \frac{\sqrt{2}}{2}, \quad \tau (q,r) = \arccos(2/3), \quad \tau (p,r) \approx 1.5749869992354. 
\end{equation*}
If $m = \gamma_{qr}(s_r/2)$ is the midpoint of the side $[q,r]$ then 
\begin{equation*}
\tau (q, m) = \tau (m, r) = \frac{\arccos(2/3)}{2} , \quad \text{and} \quad \tau(p,m) \approx 1.1482901642771.
\end{equation*}
\begin{figure}[h]\label{fig:triangle}
    \centering
    \includegraphics[width=0.5\linewidth]{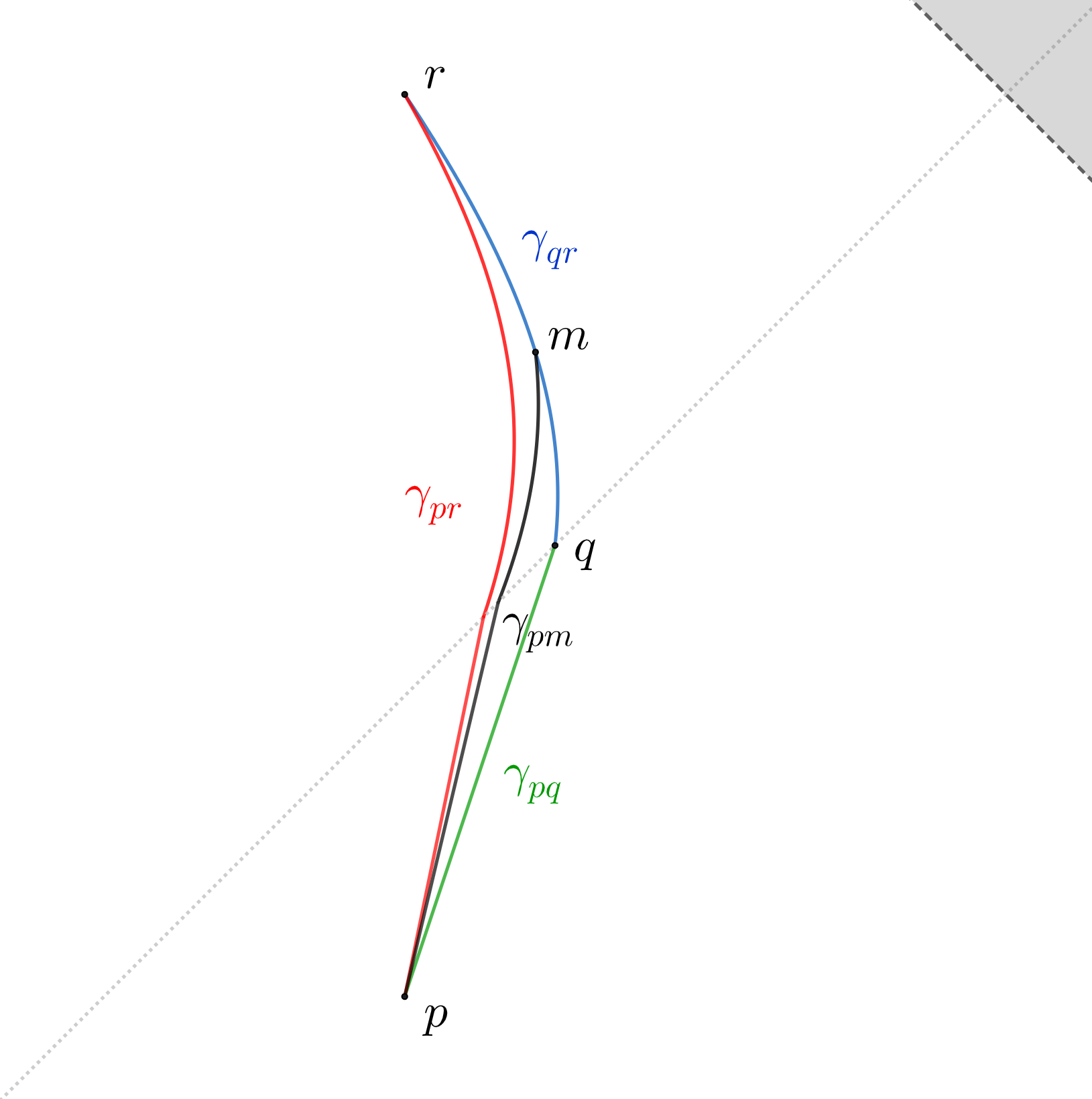}
    \caption{The timelike triangle $\Delta pqr$ in Example \ref{ex:AdS} which does not satisfy the bounds for $(\geq 0)$ curvature via one-sided triangle comparison.}
    \label{fig:ads_mink}
\end{figure}

\noindent Using \cite[Lemma 2.8 (1)]{Beran2023angles} we can compute the one-sided comparison situation for $m$ to see if we have a lower curvature bound by $K = 0$. We obtain 
\begin{equation*}
    \bar \tau(\bar p, \bar m)^2 = \frac{bd^2 + a^2 c}{b + c} - bc
\end{equation*}
where $a = \tau(p,q), b = \tau (q,m) = \tau (m,r)=c, d = \tau (p,r)$. Thus 
\begin{equation*}
    \bar\tau (\bar p, \bar m) \approx 1.1460
    < 1.1482
    \approx \tau(p,m),
\end{equation*}
which shows that this triangle does not satisfy the one-sided triangle comparison for~$(\geq 0)$-curvature bounds (see \cite[Def 3.2]{CurvBoundsLLS}). 
\end{rem}

\subsection{Further simple examples for gluing along achronal boundaries} 

In this section, we provide some additional examples of Lorentzian length spaces obtained by adding a discontinuous conformal factor on Minkowski spacetime. Note that on a smooth spacetime, the metric tensor plays the role of derivative of the length functional. Hence a discontinuous metric tensor is not necessarily pathological from the point of view of Lorentzian length spaces, where not even the existence of derivatives is assumed. Conformal transformations of Lorentzian length spaces have already been studied more in general by Manzano et al \cite{MMSZ}, but assuming continuity of the conformal factor. 
Example~\ref{ex:Mink} is showing that, although the conformal factor is discontinuous, the space that the gluing produces is essentially the Minkowski plane. Example~\ref{ex:wedge} shows that if the achronal boundary is not smooth, gluing along this boundary can mess up curvature bounds because it produces multiple maximizing curves, even though the conformal factor is constant on each region.


\begin{exam}[Gluing does not necessarily produce a new space and can preserve both upper and lower curvature bounds] \label{ex:Mink}
	Let $M := \{(t,x) : t < - x\}$ and, similar to Example \ref{ex:AdS},  choose $g := \Omega (-dt^2 + dx^2)$, with a discontinuous conformal factor, but now constant in each region, i.e.,
	\begin{equation*}
		\Omega := \begin{cases}
			1 &\text{if } \vert t \vert \geq \vert x \vert, \\
			\frac{1}{m^2} &\text{if } \vert t \vert < \vert x \vert.
		\end{cases}
	\end{equation*}
	for $m > 0$. Just as in Example \ref{ex:AdS}, $g$ induces the structure of a Lorentzian length space.
	
	
	
    We will see that this space has the same structure as the Minkowski plane. If we consider null coordinates $(u = \frac{t+x}{\sqrt2},v=\frac{t-x}{\sqrt2})$ then $M = \{(u,v) \mid u < 0\}$ and the metric is written as $g:= \Omega du dv$, where 
    \begin{equation*}
		\Omega := \begin{cases}
			1 &\text{if } v \leq 0, \\
			\frac{1}{m^2} &\text{if } v > 0.
		\end{cases}
	\end{equation*}
    In that case, it is easy to see that the map $F: (M, g) \to (M, g_0)$, given by
    \begin{equation*}
        F(u,v) := \begin{cases}
			(u,v) &\text{if } v \leq 0, \\
			(u,\frac{1}{m^2}v) &\text{if } v > 0.
		\end{cases}
    \end{equation*}
    where $g_0 = dudv$ is the usual metric in $\bR^{1,1}$, is a $\tau$-preserving and a $\leq$-preserving map (in the sense of \cite[Definition 3.2.5]{gluing}). So our ''new'' Lorentzian length space is just the Minkowski plane again. 
    In particular, maximizers between timelike related points exist, are unique and are timelike, and 
    given any timelike triangle $x \ll y \ll z$ and a point $p$ on one side of the triangle with opposite vertex $v \in \{x, y, z\}$ we have that 
    \[ \tau (p,v) = \bar \tau_0 (\bar p, \bar v), \quad \text{and} \quad \tau (v,p) = \bar\tau_0 (\bar v, \bar p),\]
    which, by one sided triangle comparison \cite[Def 3.2 and Prop 3.3]{CurvBoundsLLS}, implies that $M$ has timelike curvature bounded above and below by $K = 0$.
\end{exam}

\begin{exam}[Gluing does not have to preserve upper curvature bounds] \label{ex:wedge} 
Take $M = \bR^{1,1}$ with $(t,x)$ coordinates and note that $\Sigma = \{t = |x|\}$ determines a achronal boundary consisting of two lightlike rays and divides $\bR^2$ into two regions. Consider the metric $g := \Omega (-dt^2 + dx^2)$, with discontinuous conformal factor
\begin{equation*}
    \Omega := \begin{cases}
			1 &\text{if } t \leq |x|, \\
            c^2 &\text{if } t > |x|,
		\end{cases}
\end{equation*}
where $c > 0$ is a constant. 

\begin{figure}[ht]
		\centering
		\begin{tikzpicture}[thick,scale=1]
            \draw[dashed, gray] (0,-2.6) -- (0, 1.6);
			\draw[dashed, red] (-1.5,1.5) -- (0,0); 
            \draw[dashed, red] (0,0) --  (1.5,1.5); 
			
			\draw[-, blue] (0,-2.5) -- node[right]{$\gamma$} (0.65, 0.65); 
			\draw[-, teal] (0,-2.5) -- node[left]{$\gamma^-$} (-0.65,0.65); 
	       	\draw[-, blue] (0.65,0.65) -- (0,1.6); 
            \draw[-, teal] (-0.65,0.65) -- (0,1.6);
            
			\draw (0,-2.65) node[below]{$p$};
            \draw (0, 1.7) node[above]{$r$};
			\draw (0.7, 0.7) node[right]{$q$};
            \draw (-0.7, 0.7) node[left]{$q^-$};
            \draw (1.5, 1.6) node[right]{$\{t = |x|\}$};
		\end{tikzpicture}
		\caption{Two maximizing timelike geodesics $\gamma$ and $\gamma^-$ from $p$ to $r$ with $x_p = x_r = 0$ crossing the interface $\{ t = |x|\}$ at points $q$ and $q^-$, respectively.} \label{comparison_triangles}
	\end{figure}
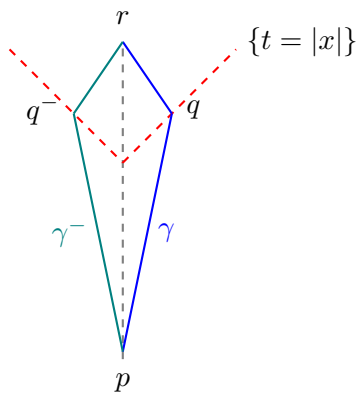

Given any $p \ll r$ with $x_p = x_r = 0$ and $t_p < 0 < t_r$, we have that the Lorentzian distance from $p$ to $r$ will be achieved by a curve $\gamma$ which is the concatenation of a geodesic contained in $\{t < |x|\}$ and one contained in $\{t > |x|\}$. Thus, its length is given by
\begin{equation*}
    \ell (x) = \sqrt{(|x| - t_p)^2 - x^2} + c \sqrt{(t_r - |x|)^2 - x^2},
\end{equation*}
where $q := (|x|, x)$ is the point where $\gamma$ intersects $\Sigma$. Note that if a maximizer $\gamma$ from $p$ to $r$ intersects $\{t = |x|\}$ at a point $q = (|x|, x)$, then we have another distance maximizer $\gamma^-$ of equal length that intersects $\Sigma$ at $q^-=(|x|, -x)$.  Such maximizers are depicted in Figure \ref{comparison_triangles} and can also be obtained from noting that $\gamma$ is also maximizing between $p$ and $r$ in the Lorentzian length space from Example \ref{ex:Mink} for $m=\frac{1}{c}$, where one can exploit the $\tau$-preserving map $F$. 
Thus $(M,g)$ does not have $(\leq 0)$-timelike curvature bounds in the sense of \cite[Def 3.2]{CurvBoundsLLS} by  \cite[Theorem~4.7]{Beran2025patchwork}. 
\end{exam}

\section{Re-examining the Penrose singularity theorem} \label{sec:timelikepenrose}

As alluded to in the introduction, one reason why the question of parametrization for null geodesics is so fundamental lies in the Penrose singularity theorem. Given the difficulties of this problem in Lorentzian length spaces outlined in the previous section, we will now take a closer look at the smooth case and see whether some version of a Penrose type theorem can be established while relying as little as possible on direct reasoning with null geodesics.

\subsection{Proof via approximation and McCann's null energy condition}

In cite \cite{McCnull}, McCann defines a synthetic null energy condition for metric-measure spacetimes by requiring a synthetic timelike lower curvature bound on each compact set. Translated to spacetimes (of sufficient regularity to define the Ricci tensor), it means the following:

\begin{defn}
\label{def:NEC}
    A spacetime $(M,g)$ satisfies the \emph{null energy condition} if for every compact set $K \subset M$, there is a constant $C_K \in \bR$ such that
   \begin{equation*}
       \Ric (v,v) \geq C_K g(v,v) \quad \text{for all $p \in K$ and all timelike } v \in T_pM 
   \end{equation*} 
\end{defn}

\noindent In the same paper, it was shown that this is equivalent to the usual null energy condition, which states that $\Ric(v,v) \geq 0$ for all $v \in TM$ null (see \cite[Cor.~27]{McCnull}).

In this section, we prove the classical Penrose singularity theorem using Def.~\ref{def:NEC}, without passing 
to the usual formulation of the null energy condition via the equivalence mentioned above. For this, we need a version of Proposition~\ref{prop:approx} where we do not fix the initial point, but instead an initial submanifold, which in the case of interest will be a trapped surface. Recall also that a trapped surface can be defined as a compact $(N-1)$-dimensional submanifold $\trap \in M$ with everywhere past-directed timelike mean curvature vector $H$ \cite[Def.~14.57]{Oneill}.

\begin{lem} \label{prop:approxtrap}
    Let $\trap \subset M$ be a compact submanifold, and fix $b > 0$. Suppose that $q \in J^+(\trap)\setminus I^+(\trap)$ and that there is a unique affinely parametrized null geodesic $\gamma \colon [0,b] \to M$ from $\trap$ to $q$. Then there exists a sequence $\gamma_n \colon [0,b] \to M$ of affinely parametrized timelike geodesics maximizing the distance to $\trap$ that converges $C^1$-uniformly to $\gamma$.
\end{lem}

\begin{proof}
    Let $q_n$ be a sequence of points converging to $q$ and choose $q_0$ such that $q \ll q_{n} \ll q_0$ for all $n$. Hence $q_n \in I^+(\trap)$, and by compactness of $\trap$ and global hyperbolicity, for each $n$ there is a timelike geodesic $\gamma_n \colon [0,b_n] \to M$ which realizes the Lorentzian distance between $\trap$ and $q_n$. Let $V_n := \dot\gamma_n(0)$ and fix an auxiliary Riemannian metric $h$. We may affinely parametrize the $\gamma_n$ such that $\Vert V_n \Vert_h = \Vert \dot\gamma(0) \Vert_h$. Moreover, by compactness of $\trap$ and of the $h$-balls in each tangent space, the sequence $(V_n)_n$ must have a convergent subsequence $(V_{n_k})_k$. Well-posedness of the IVP for the geodesic equation implies that the curves $\gamma_{n_k}$ converge $C^1$-uniformly to some limit causal geodesic $\sigma$ from $\trap$ to $q$. Note that $\sigma$ must extend all the way to $q$, because otherwise $\sigma$ would be inextendible and contained in the compact set $J^+(\sigma(0)) \cap J^-(q_0)$, a contradiction to non-total imprisonment. But then the images of $\sigma$ and $\gamma$ coincide, since we assumed that $\gamma$ is the unique causal geodesic from $\trap$ to $q$. Moreover, the parametrizations also coincide by our choice of normalization of the $V_n$, so $\gamma = \sigma$. It also follows that, in fact, the original sequence $\gamma_n$ converges to $\gamma$. Finally, we conclude that $b_n \to b$. Slightly prolonging or shortening the $\gamma_n$ as needed (which is possible, at least, for all $n$ large enough), we obtain a sequence $\gamma_n \colon [0,b] \to M$.
    \end{proof}

\begin{thm}[Penrose theorem] \label{thm:Pen}
    Let $(M,g)$ be a $(N+1)$-dimensional globally hyperbolic spacetime with non-compact Cauchy surfaces that satisfies Definition~\ref{def:NEC}. If $\trap$ is a trapped surface, then there is at least one future-incomplete null geodesic starting from $\trap$.
\end{thm}

\begin{proof}
    Let $\gamma \colon [0,b] \to M$ be a null geodesic $g$-orthogonal to $\trap$, and denote $p := \gamma(0)$. At $p$, we may choose a $g$-orthonormal basis $e_2,...,e_N$ tangent to $\trap$, and complete it to an orthonormal basis of $T_pM$ by adding a future-directed timelike vector $e_0$ and a spacelike vector $e_1$. Without loss of generality, we assume $\dot\gamma(0) = e_0 + e_1$ (this is equivalent to fixing a ``normalization" for null vectors). Moreover, we parallel transport the basis $e_0,...,e_N$ to obtain a frame field $E_0,...,E_N$ along $\gamma$. In particular, $\dot\gamma = E_0 + E_1$ along the entire curve.
    
    By Lemma~\ref{prop:approxtrap}, we can approximate $\gamma$ by maximizing timelike geodesics $\gamma_n$ starting on $\trap$ with initial tangent vectors $\dot\gamma_n(0) \to \dot\gamma(0)$. Locally, we can smoothly extend $e_0,...,e_N$ to smooth orthonormal vector fields along $\trap$ (e.g. by parallel transport in a normal neighborhood), and then parallel transport along each $\gamma_n$, obtaining frame fields $E_0^n,...,E_N^n$. By well-posedness of the IVP for parallel transport, we have that $E_j^n \to E_j$ uniformly for each $j$ as $n \to \infty$.
    
    Let $f(s) = 1-s/b$ for $s \in [0,b]$. Consider the formula for the second variation $L''^n_j$ of the length of $\gamma_n$ with $f E^n_j$ as the variation vector field,
    \begin{equation*}
        g(\dot\gamma_n(0),\dot\gamma_n(0)) L''^n_{j} = \frac{1}{b}- \int_0^b f^2 \Riem(E^n_j,\dot\gamma_n,E^n_j,\dot\gamma_n) ds - {g(\dot\gamma_n(0),\mathrm{II}(e_j,e_j))}.
    \end{equation*}
    Here $\mathrm{II}$ denotes the second fundamental form of $\trap$. Summing over $j$ we obtain

 \begin{multline*}
       I_n := g(\dot\gamma_n(0),\dot\gamma_n(0)) \sum_{j=2}^{N} L''^n_{j} = \frac{N-1}{b} - \int_0^b f^2 \Ric(\dot\gamma_n,\dot\gamma_n) ds - {g(\dot\gamma_n(0),H)}\\ - \int_0^b f^2 \left( \Riem(E^n_0,\dot\gamma_n,E^n_0,\dot\gamma_n) - \Riem(E^n_1,\dot\gamma_n,E^n_1,\dot\gamma_n) \right) ds,
    \end{multline*}
    where $H = \sum_{j=2}^N \mathrm{II}(e_j,e_j)$ is the mean curvature vector of $\trap$, which is past-directed timelike by the trappedness assumption. We want to show that $I_n < C < 0$ for all large enough $n$, hence we need to study the limiting behavior of each term. It is straightforward that
    \begin{align*}
        \lim_{n \to \infty} g(\dot\gamma_n(0),H) &= g(\dot\gamma(0),H) = g(e_0+e_1,H), &
        \lim_{n \to \infty} \dot\gamma_n &= \dot\gamma = E_0 + E_1,
    \end{align*}
    where in the second equation, the convergence is uniform in the curve parameter $s \in [0,b]$. Since we can find a compact neighborhood of $\gamma$ that will contain all $\gamma_n$'s for large enough $n$, Definition~\ref{def:NEC} implies
    \begin{equation*}
        \liminf_{n \to \infty} \Ric(\dot\gamma_n,\dot\gamma_n) \geq \lim_{n \to \infty} C g(\dot\gamma_n,\dot\gamma_n) = C g(\dot\gamma,\dot\gamma) = 0,
    \end{equation*}
    where again the convergence is uniform. Finally, we also have
   \begin{equation*}
        \lim_{n \to \infty} \Riem(E^n_0,\dot\gamma_n,E^n_0,\dot\gamma_n) = \Riem(E_0,E_1,E_0,E_1) = \lim_{n \to \infty} \Riem(E^n_1,\dot\gamma_n,E^n_1,\dot\gamma_n)
    \end{equation*}
   uniformly. It follows that
   \begin{equation*}
       \limsup_{n \to \infty} I_n = \frac{N-1}{b} - {g(e_0+e_1,H)}.
    \end{equation*}
    Because $\trap$ is trapped, $H$ is past-directed timelike. It follows that $g(e_0+e_1,H) > 0$, and hence, for $b$ large enough, $\limsup_{n \to \infty} I_n < 0$.

    Assume that $\gamma$ can be extended to a future complete null geodesic which maximizes the distance to $\trap$. Hence we may choose $b$ large enough so that $\limsup_{n \to \infty} I_n < 0$ and hence $I_n < 0$ for all large enough $n$. This is in contradiction to the fact that the $\gamma_n$ maximize the distance to $\trap$. We conclude that every complete null geodesic starting from $\trap$ stops maximizing the distance after some parameter value. The usual causality argument (which does not rely on the null energy condition) tells us that every geodesic being complete is incompatible with the assumption of a non-compact Cauchy surface (see c.f \cite[Coro 14.61 (A)]{Oneill}). Therefore, there is at least one incomplete null geodesic, concluding the proof.
\end{proof}


In the proof, we have used continuity of $\Riem$, which is only guaranteed for $g \in C^2$. It would suffice, however, to have certain bounds on those terms, depending on $H$.

\subsection{Obtaining timelike incompleteness}

Working with timelike geodesics, as above, one can also wonder if it is possible to obtain an incomplete timelike geodesic. The theorem below achieves this, provided some stronger curvature conditions are met.

\begin{thm}
    Let $(M,g)$ be a $(N+1)$-dimensional globally hyperbolic spacetime that satisfies the strong energy condition $\Ric(v,v) \geq 0$ for all timelike $v \in TM$. Furthermore, let $\trap \subset M$ be a trapped surface, with past-directed timelike mean curvature vector $H$. Suppose that there exists a Riemannian metric $h$ on $M$ such that for every future-directed timelike geodesic $\sigma$ $g$-normal to $\trap$ with $h(\dot\sigma(0),\dot\sigma(0)) = 1$ it holds that
    \begin{equation} \label{eq:initialcond}
      1+\int_0^{1} \left(1-u^2 \right) \Riem(\dot\sigma,E_1,\dot\sigma,E_1) du < g(\dot\sigma(0),H),
     \end{equation}
     provided that $\sigma$ is defined on $[0,1]$. Here $E_1$ is the parallel transport along $\sigma$ of the unique (up to sign) spacelike $g$-unit vector $g$-orthogonal to $\dot\sigma(0)$ and $\trap$. Then there is a constant $\ell >0$ such that all future-directed timelike curves starting on $\trap$ have Lorentzian $g$-length at most $\ell$. In particular, $(M,g)$ is future timelike geodesically incomplete.
\end{thm}

This theorem improves upon \cite{LeoPRD} because here assumption \eqref{eq:initialcond} is made on a relatively compact set of vectors (whose image under the exponential map is relatively compact in $M$). Thus the assumption has more accurately the character of an initial condition that only needs to be checked close to $\trap$, rather than a global assumption on $M$. As discussed in \cite{LeoPRD}, such an assumption has an interpretation in terms of the radial tidal forces being repulsive. Note also that every manifold admits many Riemannian metrics, so the existence of $h$ is not a restriction per se, but just a convenient way to formulate the existence of a relatively compact set in the normal bundle to $\trap$ with the required property \eqref{eq:initialcond}. 

\begin{proof}
    Our goal is to show that every future-directed timelike geodesic $\sigma$ normal to $\trap$ with $h(\dot\sigma(0),\dot\sigma(0)) = 1$ encounters a focal point after some finite affine parameter at most $b$ independent of $\sigma$. Then, since the $h$-unit normal bundle to $\trap$ is compact, $g(\dot\sigma(0),\dot\sigma(0))$ and hence $L_g(\sigma \vert_{[0,b]})$ are uniformly bounded above. Notice that by global hyperbolicity and compactness of $\trap$, for every $p \in I^+(\trap)$, there exists a geodesic realizing the Lorentzian distance $d(\trap,p)$. This distance is thus bounded, implying the claim of the theorem.

    To show the existence of a suitable affine parameter, let
    \begin{align*}
     \varphi(u) &:= \begin{cases} u &\text{for } u \in [0,1], \\ 1-\frac{u-1}{b-1} &\text{for } u \in (1,b],\end{cases}\\
     \psi(u) &:= \begin{cases} 1 &\text{for } u \in [0,1], \\ 1-\frac{u-1}{b-1} &\text{for } u \in (1,b].\end{cases}
    \end{align*}
    The second variation of the length of $\sigma$ with respect to $\varphi E_1$ is given by
    \begin{equation*}
        g(\dot\sigma(0),\dot\sigma(0)) L''_{1} = 1+\frac{1}{b-1}- \int_0^b \varphi^2 \Riem(E_1,\dot\sigma,E_1,\dot\sigma) du - {g(\dot\sigma(0),\mathrm{II}(e_1,e_1))}.
    \end{equation*}
    Similarly, for variation vector field $\psi E_j$, $j=2,...,N$, we obtain
    \begin{equation*}
        g(\dot\sigma(0),\dot\sigma(0)) L''_{j} = \frac{1}{b-1}- \int_0^b \psi^2 \Riem(E_j,\dot\sigma,E_j,\dot\sigma) du - {g(\dot\sigma(0),\mathrm{II}(e_j,e_j))}.
    \end{equation*}
    Summing up, we obtain
    \begin{align*}
        I &:= g(\dot\sigma(0),\dot\sigma(0)) \sum_{j=1}^N L''_{j} \\
        &= 1+ \frac{N}{b-1} - g(\dot\sigma(0),H) + \int_0^{1} \left(\psi^2-\varphi^2 \right) \Riem(\dot\sigma,E_1,\dot\sigma,E_1) du - \int_0^{b} \psi^2 \Ric(\dot\sigma,\dot\sigma) du \\
        &= 1+ \frac{N}{b-1} - g(\dot\sigma(0),H) + \int_0^{1} \left(1-u^2 \right) \Riem(\dot\sigma,E_1,\dot\sigma,E_1) du - \int_0^{b} \psi^2 \Ric(\dot\sigma,\dot\sigma) du.
    \end{align*}
    Combining the SEC and \eqref{eq:initialcond}, we deduce that for $b>1$ large enough, $I<0$, implying the existence of a focal point (notice that $g(\dot\sigma(0),\dot\sigma(0)) < 0$, so the sign of $I$ is the opposite of 
    the sign of the second variation of the length). 
\end{proof}

\section*{Acknowledgements}

We thank the organizers and participants of the ``Kick-off workshop: A new geometry for Einstein’s theory of relativity and beyond” held at the University of Vienna, where an early version of this research was discussed, and Jiří Podolský in particular for pointing out that one half of Example~\ref{ex:AdS} is isometric to a region in $AdS$.

MG acknowledges support by the Deutsche Forschungsgemeinschaft (DFG, German Research Foundation) under Germany’s Excellence Strategy -- EXC 2121 ''Quantum Universe'' -- 390833306. SB acknowledges the support by the IMAG-María de Maeztu grant CEX2020-001105-MCIN/AEI/10.13039/501100011033 as well as the project PID2024-156031NB-I00 funded by MICIU/AEI/10.13039/{501100011033/ERDF/EU. 
This work was funded in part by the Austrian Science Fund (FWF) grant 10.55776/EFP6. For open access purposes, the authors have applied a CC-BY public copyright license to any author accepted manuscript version arising from this submission. 

\bibliographystyle{abbrv}
\bibliography{refs}

@article{LangeLytchakSaemann,
 author = {Lange, Christian and Lytchak, Alexander and Sämann, Clemens},
 title = {Lorentz meets Lipschitz},
 fjournal = {Advances in Theoretical and Mathematical Physics},
 journal = {Adv. Theor. Math. Phys.},
 issn = {8},
 volume = {25},
 pages = {2141--2170},
 year = {2021}
}

@article{Beran2025patchwork,
 author = {Beran, Tobias and Napper, Lewis and Rott, Felix},
 title = {Alexandrov's patchwork and the {Bonnet}-{Myers} theorem for {Lorentzian} length spaces},
 fjournal = {Transactions of the American Mathematical Society},
 journal = {Trans. Am. Math. Soc.},
 issn = {0002-9947},
 volume = {378},
 number = {4},
 pages = {2713--2743},
 year = {2025},
 language = {English},
 doi = {10.1090/tran/9372},
 zbMATH = {8013998},
 Zbl = {1564.53065}
}

@article {AdS,
    AUTHOR = {Ball\'on Bayona, C. A. and Braga, Nelson R. F.},
     TITLE = {Anti-de {S}itter boundary in {P}oincar\'e{} coordinates},
   JOURNAL = {Gen. Relativity Gravitation},
  FJOURNAL = {General Relativity and Gravitation},
    VOLUME = {39},
      YEAR = {2007},
    NUMBER = {9},
     PAGES = {1367--1379},
      ISSN = {0001-7701,1572-9532},
   MRCLASS = {83C15},
  MRNUMBER = {2329041},
       DOI = {10.1007/s10714-007-0446-y},
       URL = {https://doi.org/10.1007/s10714-007-0446-y},
}

@article{Beran2023angles,
 author = {Beran, Tobias and S{\"a}mann, Clemens},
 title = {Hyperbolic angles in {Lorentzian} length spaces and timelike curvature bounds},
 fjournal = {Journal of the London Mathematical Society. Second Series},
 journal = {J. Lond. Math. Soc., II. Ser.},
 issn = {0024-6107},
 volume = {107},
 number = {5},
 pages = {1823--1880},
 year = {2023},
 language = {English},
 doi = {10.1112/jlms.12726},
 zbMATH = {7731089},
 Zbl = {1522.53025}
}

@article{CRHK17,
  title = {Exact geodesic distances in {FLRW} spacetimes},
  author = {Cunningham, W. J. and Rideout, D. and Halverson, J. and Krioukov, D.},
  journal = {Phys. Rev. D},
  volume = {96},
  issue = {10},
  pages = {103538},
  numpages = {11},
  year = {2017},
  month = {Nov},
  publisher = {American Physical Society},
  doi = {10.1103/PhysRevD.96.103538},
  url = {https://link.aps.org/doi/10.1103/PhysRevD.96.103538}
}

@misc{MMSZ,
      title={Conformal transformations of metric spaces and Lorentzian pre-length spaces}, 
      author={Miguel Manzano and Karim Mosani and Clemens Sämann and Omar Zoghlami},
      year={2025},
      howpublished = {Preprint arXiv:2512.05842},
      eprint={2512.05842},
      archivePrefix={arXiv},
      primaryClass={math.DG},
      url={https://arxiv.org/abs/2512.05842}, 
}

@book{Oneill,
 author = {O'Neill, B.},
 title = {Semi-Riemannian geometry. {With} applications to relativity},
 fseries = {Pure and Applied Mathematics (Academic Press)},
 series = {Pure Appl. Math., Academic Press},
 issn = {0079-8169},
 volume = {103},
 year = {1983},
 publisher = {Academic Press, New York, NY},
 language = {English},
 zbMATH = {3842680},
 Zbl = {0531.53051}
}

@article{gluing,
 author = {Beran, Tobias and Rott, Felix},
 title = {Gluing constructions for {Lorentzian} length spaces},
 fjournal = {Manuscripta Mathematica},
 journal = {Manuscr. Math.},
 issn = {0025-2611},
 volume = {173},
 number = {1-2},
 pages = {667--710},
 year = {2024},
 language = {English},
 doi = {10.1007/s00229-023-01469-4},
 zbMATH = {7785305},
 Zbl = {1543.53040}
}

@article{AlBi,
 author = {Alexander, Stephanie B. and Bishop, Richard L.},
 title = {Lorentz and semi-{Riemannian} spaces with {Alexandrov} curvature bounds},
 fjournal = {Communications in Analysis and Geometry},
 journal = {Commun. Anal. Geom.},
 issn = {1019-8385},
 volume = {16},
 number = {2},
 pages = {251--282},
 year = {2008},
 language = {English},
 doi = {10.4310/CAG.2008.v16.n2.a1},
 zbMATH = {5317396},
 Zbl = {1149.53040}
}

@book{BEE,
	author = {Beem, John K. and Ehrlich, Paul E. and Easley, Kevin L.},
	edition = {Second},
	isbn = {0-8247-9324-2},
	mrclass = {53C50 (53-02 83-02)},
	mrnumber = {1384756},
	mrreviewer = {Peter R. Law},
	publisher = {Marcel Dekker, Inc., New York},
	series = {Monographs and Textbooks in Pure and Applied Mathematics},
	title = {Global {L}orentzian geometry},
	volume = {202},
	year = {1996}
}

@article {CaMo,
    AUTHOR = {Cavalletti, Fabio and Mondino, Andrea},
     TITLE = {Optimal transport in {L}orentzian synthetic spaces, synthetic
              timelike {R}icci curvature lower bounds and applications},
   JOURNAL = {Camb. J. Math.},
  FJOURNAL = {Cambridge Journal of Mathematics},
    VOLUME = {12},
      YEAR = {2024},
    NUMBER = {2},
     PAGES = {417--534},
      ISSN = {2168-0930,2168-0949},
   MRCLASS = {53C23 (49Q22 53C50 53C80 83C75)},
  MRNUMBER = {4779676},
MRREVIEWER = {Nicola\ Gigli},
       DOI = {10.4310/cjm.2024.v12.n2.a3},
       URL = {https://doi.org/10.4310/cjm.2024.v12.n2.a3},
}

@article {CMMsmooth,
    AUTHOR = {Cavalletti, Fabio and Manini, Davide and Mondino, Andrea},
     TITLE = {Optimal transport on null hypersurfaces and the null energy
              condition},
   JOURNAL = {Comm. Math. Phys.},
  FJOURNAL = {Communications in Mathematical Physics},
    VOLUME = {406},
      YEAR = {2025},
    NUMBER = {9},
     PAGES = {Paper No. 212, 62},
      ISSN = {0010-3616,1432-0916},
   MRCLASS = {49Q22 (53C50 83C05 83C75)},
  MRNUMBER = {4940202},
       DOI = {10.1007/s00220-025-05345-y},
       URL = {https://doi.org/10.1007/s00220-025-05345-y},
}

@misc{CMMsynth,
      title={On the geometry of synthetic null hypersurfaces}, 
      author={Fabio Cavalletti and Davide Manini and Andrea Mondino},
      year={2025},
      howpublished = {Preprint arXiv:2506.04934},
      eprint={2506.04934},
      archivePrefix={arXiv},
      primaryClass={math.DG},
      url={https://arxiv.org/abs/2506.04934}, 
}

@article{CurvBoundsLLS,
	author = {Beran, Tobias and Kunzinger, Michael and Rott, Felix},
	title = {On curvature bounds in {Lorentzian} length spaces},
	fjournal = {Journal of the London Mathematical Society. Second Series},
	journal = {J. Lond. Math. Soc., II. Ser.},
	issn = {0024-6107},
	volume = {110},
	number = {2},
	pages = {41},
	note = {Id/No e12971},
	year = {2024},
	language = {English},
	doi = {10.1112/jlms.12971},
	zbMATH = {7900392},
	Zbl = {1547.53080}
}

@article {HawArea,
    AUTHOR = {Hawking, S. W.},
     TITLE = {Black holes in general relativity},
   JOURNAL = {Comm. Math. Phys.},
  FJOURNAL = {Communications in Mathematical Physics},
    VOLUME = {25},
      YEAR = {1972},
     PAGES = {152--166},
      ISSN = {0010-3616,1432-0916},
   MRCLASS = {83.53},
  MRNUMBER = {293962},
MRREVIEWER = {H.\ Rund},
       URL = {http://projecteuclid.org/euclid.cmp/1103857884},
}

@article{HawSing,
 author = {Hawking, Stephen W.},
 title = {The occurrence of singularities in cosmology},
 fjournal = {Proceedings of the Royal Society of London. Series A. Mathematical and Physical Sciences},
 journal = {Proc. R. Soc. Lond., Ser. A},
 issn = {0080-4630},
 volume = {294},
 pages = {511--521},
 year = {1966},
 language = {English},
 doi = {10.1098/rspa.1966.0221},
 zbMATH = {3226977},
 Zbl = {0139.45803}
}

@article{Kett,
 author = {Ketterer, Christian},
 title = {Characterization of the null energy condition via displacement convexity of entropy},
 fjournal = {Journal of the London Mathematical Society. Second Series},
 journal = {J. Lond. Math. Soc., II. Ser.},
 issn = {0024-6107},
 volume = {109},
 number = {1},
 pages = {24},
 note = {Id/No e12846},
 year = {2024},
 language = {English},
 doi = {10.1112/jlms.12846},
 zbMATH = {7809019},
 Zbl = {1535.83087}
}

@article {KuSa,
    AUTHOR = {Kunzinger, Michael and S\"amann, Clemens},
     TITLE = {Lorentzian length spaces},
   JOURNAL = {Ann. Global Anal. Geom.},
  FJOURNAL = {Annals of Global Analysis and Geometry},
    VOLUME = {54},
      YEAR = {2018},
    NUMBER = {3},
     PAGES = {399--447},
      ISSN = {0232-704X,1572-9060},
   MRCLASS = {53C23 (53B30 53C50 53C80)},
  MRNUMBER = {3867652},
MRREVIEWER = {Benjam\'in\ Olea},
       DOI = {10.1007/s10455-018-9633-1},
       URL = {https://doi.org/10.1007/s10455-018-9633-1},
}

@article{LeoPRD,
    author = "Garc\'\i{}a-Heveling, Leonardo",
    title = "{Radial gravitational collapse causes timelike incompleteness}",
    eprint = "2401.16297",
    archivePrefix = "arXiv",
    primaryClass = "gr-qc",
    doi = "10.1103/PhysRevD.109.084034",
    journal = "Phys. Rev. D",
    volume = "109",
    number = "8",
    pages = "084034",
    year = "2024"
}

@article {MelanieC1,
    AUTHOR = {Graf, Melanie},
     TITLE = {Singularity theorems for {$C^1$}-{L}orentzian metrics},
   JOURNAL = {Comm. Math. Phys.},
  FJOURNAL = {Communications in Mathematical Physics},
    VOLUME = {378},
      YEAR = {2020},
    NUMBER = {2},
     PAGES = {1417--1450},
      ISSN = {0010-3616,1432-0916},
   MRCLASS = {53C50 (53C20 83C75)},
  MRNUMBER = {4134950},
MRREVIEWER = {Clemens\ Saemann},
       DOI = {10.1007/s00220-020-03808-y},
       URL = {https://doi.org/10.1007/s00220-020-03808-y},
}

@article {McCnull,
    AUTHOR = {McCann, Robert J.},
     TITLE = {A synthetic null energy condition},
   JOURNAL = {Comm. Math. Phys.},
  FJOURNAL = {Communications in Mathematical Physics},
    VOLUME = {405},
      YEAR = {2024},
    NUMBER = {2},
     PAGES = {Paper No. 38, 24},
      ISSN = {0010-3616,1432-0916},
   MRCLASS = {51K10 (53C50 83C75)},
  MRNUMBER = {4703452},
MRREVIEWER = {Barry\ Minemyer},
       DOI = {10.1007/s00220-023-04908-1},
       URL = {https://doi.org/10.1007/s00220-023-04908-1},
}

@article{PenSing,
 author = {Penrose, Roger},
 title = {Gravitational collapse and space-time singularities},
 fjournal = {Physical Review Letters},
 journal = {Phys. Rev. Lett.},
 issn = {0031-9007},
 volume = {14},
 pages = {57--59},
 year = {1965},
 language = {English},
 doi = {10.1103/PhysRevLett.14.57},
 zbMATH = {3203444},
 Zbl = {0125.21206}
}

@article {SaSt,
    AUTHOR = {S\"amann, Clemens and Steinbauer, Roland},
     TITLE = {On geodesics in low regularity},
   JOURNAL = {J. Phys. Conf. Ser.},
  FJOURNAL = {Journal of Physics. Conference Series},
    VOLUME = {968},
      YEAR = {2018},
     PAGES = {012010, 14},
      ISSN = {1742-6588,1742-6596},
   MRCLASS = {53C22 (34A36 53B30 83C10)},
  MRNUMBER = {3919953},
MRREVIEWER = {Wolfgang\ Hasse},
       DOI = {10.1088/1742-6596/968/1/012010},
       URL = {https://doi.org/10.1088/1742-6596/968/1/012010},
}

\end{document}